\documentclass[reqno]{amsart}%
\usepackage{amstext}
\usepackage{amsfonts}
\usepackage{amsmath}
\usepackage{amssymb}
\usepackage{hyperref}
\usepackage{graphicx}%
\providecommand{\U}[1]{\protect\rule{.1in}{.1in}}
\numberwithin{equation}{section}
\newtheorem{theorem}{Theorem}[section]

\newtheorem{proposition}[theorem]{Proposition}
\newtheorem{corollary}[theorem]{Corollary}
\newtheorem{lemma}[theorem]{Lemma}

\begin{document}
\title[Upper semicontinuous hyperviscous flow]{Upper semicontinuous attractors for 3D hyperviscous flow}
\author{Abdelhafid Younsi}
\address{Department of Mathematics and Computer Science , University of Djelfa , Algeria.}
\email{younsihafid@gmail.com}

\begin{abstract}
We regularized the 3D Navier-Stokes equations by adding a high-order viscosity
term. We first prove the existence of the global attractors of the Leray-Hopf
weak solutions of the regularized 3D Navier-Stokes equations and then we study
the upper semicontinuity, as the artificial dissipation $\varepsilon$ goes to
$0$. We also consider applications of obtained results to the regularized
problem by allowing the family of forcing functions to vary with $\varepsilon
$, for $\varepsilon$ $>0$.

\end{abstract}
\maketitle

\section{Introduction}

In this paper, we study the robustness, or upper semicontinuity of the global
attractors of the Leray-Hopf weak solutions of modified three dimensional
Navier-Stokes equations. We regularized the 3D Navier-Stokes system by adding
a high order artificial viscosity term to the conventional system%
\begin{equation}%
\begin{array}
[c]{c}%
\dfrac{\partial u^{\varepsilon}}{\partial t}+\varepsilon(-\triangle
)^{l}u^{\varepsilon}-\nu\triangle u^{\varepsilon}+\left(  u^{\varepsilon
}.\nabla\right)  u^{\varepsilon}+\nabla p=f\left(  x\right)  ,\text{ in
}\Omega\times\left(  0,\infty\right) \\
\text{div }u^{\varepsilon}=0,\text{ in }\Omega\times\left(  0,\infty\right)
\text{, }u^{\varepsilon}\left(  x,0\right)  =u_{0}^{\varepsilon},\text{ in
}\Omega\text{, \ }\\
p(x+Le_{i},t)=p(x,t),\text{ }u^{\varepsilon}(x+Le_{i},t)=u^{\varepsilon
}(x,t)\text{\ \ }i=1,2\text{ },3,\ t\in\left(  0,\infty\right)
\end{array}
\label{1}%
\end{equation}
where $\Omega=(0,L)^{3}$ with periodic boundary conditions and $\left(
e_{1},...,e_{d}\right)  $ is the natural basis of $%
\mathbb{R}
^{d}$. Here $\varepsilon>0$ is the artificial dissipation parameter,
$u^{\varepsilon}\ $is the velocity vector field, $p$ is the pressure, $\nu>0$
is the kinematic viscosity of the fluid and $f$ is a given force field. For
$\varepsilon=0$, the model is reduced to 3D Navier--Stokes system.

Mathematical model for such fluid motion has been used extensively in
turbulence simulations (see e.g. \cite{4, 5, 10}). For further discussion of
theoretical results concerning (%
\ref{1}
), see \cite{1, 2, 6, 13, 16, 17, 20,23}.

In the work \cite{23}, the strong convergence of the solution of this problem
to the solution of the conventional system as the regularization parameter
goes to zero, was established for each dimension $d\leq4$.

For the 3D Navier--Stokes system weak solutions of problem are known to exist
by a basic result by J. Leray from 1934 \cite{11}, only the uniqueness of weak
solutions remains as an open problem. Then the known theory of global
attractors of infinite dimensional dynamical systems is not applicable to the
3D Navier--Stokes system.

The theory of trajectory attractors for evolution partial differential
equations was developed in \cite{15, 18}, which the uniqueness theorem of
solutions of the corresponding initial-value problem is not proved yet, e.g.
for the 3D Navier--Stokes system (see \cite{8a, 15, 17a, 18}). Such trajectory
attractor is a classical global attractor but in the space of weak solutions
defined on $[0,$ $\infty)$, with the corresponding semigroup being simply the
translation in time of such solutions. A compact set $\mathfrak{A}\Subset E$
is said to be a global attractor of a semigroup $\{S(t),t>0\}$ acting in a
Banach or Hilbert space $E$ if $\mathfrak{A}$ is strictly invariant with
respect to $\{S(t)\}:S(t)\mathfrak{A}=\mathfrak{A}$ $\forall t\geq0$ and
$\mathfrak{A}$ attracts any bounded set $B\subset E:dist(S(t)B$,
$\mathfrak{A})\rightarrow0$ ($t\rightarrow\infty$) (see \cite{14}, \cite{15},
\cite{17a}, \cite{18}, \cite{20}).

In this article, we study the upper semicontinuity, of the global attractors
of the Leray-Hopf weak solutions of a regularized 3D Navier-Stokes equations,
as the artificial dissipation $\varepsilon$ goes to $0$. While there exist
other examples of such robustness in the literature of the Navier-Stokes
equations, the specific emphasis on the regularized problem is new for the 3D
Navier-Stokes equations and is of interest. This would bean extension of the
earlier work on Ou and Sritharan for the 2D Navier-Stokes equations, see
references \cite{16} and \cite{17}. It is now known that there is a global
attractor $\mathfrak{A}_{0}$ for the Leray-Hopf weak solutions of the 3D
Navier-Stokes equations, see Sell \cite{17a, 18} and Chepyzhov \cite{13a}.

The main object of this paper to show that there is a global attractor, which
one might denote by $\mathfrak{A}_{\varepsilon}$, for the regularized problem
(%
\ref{1}
), and that the family \{$\mathfrak{A}_{\varepsilon}$\} is upper
semicontinuous at $\varepsilon=0$. Moreover, we can modify the argument
described above so that the final result will have broader applicability by
allowing the family of forcing functions $f^{\varepsilon}$ to vary with
$\varepsilon$, for $\varepsilon>0$.

The family of sets $\mathfrak{A}_{\varepsilon}$, $0<\varepsilon\leq1$ is
robust at $\mathfrak{A}_{0}$, or is upper semicontinuous with respect to
$\varepsilon$ at $\varepsilon_{0}=0$, provided that, for every $\varepsilon
_{0}>0$, there is a neighborhood $O\left(  \varepsilon_{0}\right)  $ of $0\in%
\mathbb{R}
$ and a neighborhood $N_{\varepsilon_{0}}(\mathfrak{A}_{0})\ $of
$\mathfrak{A}_{0}$, such that $\mathfrak{A}_{\varepsilon}\subset
N_{\varepsilon_{0}}(\mathfrak{A}_{0})$, for every $\varepsilon\in O\left(
\varepsilon_{0}\right)  $ with $\varepsilon>0$, see (23.13) in \cite{18}.

The paper is organized as follows. In Section 2, we present the relevant
mathematical framework for the paper. In Section 3, we recall the definition
of the trajectory attractor $\mathfrak{A}_{0}$ of the conventional 3-D
Navier-Stokes equations. In Section 4, we study the regularized problem (see
equation (%
\ref{1}
)), then we show the existence of trajectory attractor $\mathfrak{A}%
_{\varepsilon}$. In Section 5, we present the main result of this paper, that
is, a theorem on the upper semicontinuity on the attractors $\mathfrak{A}%
_{\varepsilon}$. Finally, an application of our general results to the study
of the robustness of the system (%
\ref{1}
)\ with a perturbed external force.

\section{Preliminary}

We denote by $H_{per}^{m}\left(  \Omega\right)  $, the Sobolev space of
$L$-periodic functions endowed with the inner product
\[
\left(  u,v\right)  =%
{\textstyle\sum\limits_{\left\vert \alpha\right\vert \leq m}}
(D^{\alpha}u,D^{\alpha}v)_{L^{2}\left(  \Omega\right)  }\text{ and the norm
}\left\Vert u\right\Vert _{m}=%
{\textstyle\sum\limits_{\left\vert \alpha\right\vert \leq m}}
(\left\Vert D^{\alpha}u\right\Vert _{L^{2}\left(  \Omega\right)  }^{2}%
)^{\frac{1}{2}}.
\]
We define the spaces $V_{m}$ as completions of smooth, divergence-free,
periodic, zero-average functions with respect to the $H_{per}^{m}$ norms.
$V_{m}^{\prime}$ denote the dual space of $V_{m}$ and $V$ denote the space
$V_{0}$.

We present the topology to be used for generating the neighborhood\ of
robustness. Let $F$ any vector space. A metric $d\left(  f,g\right)  $ on $F$
is said to be invariant if%
\[
d\left(  f,g\right)  =d\left(  f-g,0\right)  \text{ for all }f\text{, }g\in
F.
\]

A Fr\'{e}chet space is a complete topological vector space whose topology is
induced by a translation invariant metric $d\left(  f,g\right)  $. Given a
Banach space $X$,\ with norm $\left\Vert .\right\Vert _{X}$ and $1\leq
p<\infty$ , we denote by $L_{loc}^{p}\left[  0,\infty;X\right)  $ the
Fr\'{e}chet space of mesurable functions $f:\left[  0,\infty\right)
\rightarrow X$ that are $p$-integrable over $\left[  0,T\right]  $,\ for each
$0<T<\infty$ , endow with the metric%
\[
d\left(  f,g\right)  =%
{\displaystyle\sum\limits_{n=1}^{\infty}}
2^{-n}\min(\left\Vert f-g\right\Vert _{L^{p}\left(  0,\text{ }n;\text{
}X\right)  },1).
\]
We denote by $L_{loc}^{p}\left(  0,\infty;X\right)  $ the Fr\'{e}chet space of
mesurable functions $f:\left(  0,\infty\right)  \rightarrow\nolinebreak X$
that are $p$-integrable over $\left[  t_{0},T\right]  $,\ for each
$0<t_{0}\leq T<\infty$ endow with the metric%
\[
d\left(  f,g\right)  =%
{\displaystyle\sum\limits_{n=2}^{\infty}}
2^{-n}\min(\left\Vert f-g\right\Vert _{L^{p}\left(  \frac{1}{n},\text{
}n;\text{ }X\right)  },1).
\]
Similarly for $p=\infty$, we will let $L_{loc}^{\infty}\left(  0,\infty
;X\right)  $ denote the collection of all functions $f:(0,\infty)\rightarrow
X$ with the property that, for all $\tau$ and $T$ with $0<T<\infty$ , one has
$\operatorname*{ess}\sup\limits_{0<s<T}\left\Vert f\right\Vert _{X}<\infty$.
We denote by $C\left[  0,\infty;X\right)  $ the space of strongly continuous
functions from $\left[  0,\infty\right)  $ to $X$, endowed with the topologie
of the uniform convergence over compact sets and by $C_{w}\left[
0,\infty;X\right)  $ the space of weakly continuous functions from $\left[
0,\infty\right)  $ to $X$. We denote by $L^{\infty}C=L^{\infty}\left(
\mathbb{R}
,X\right)  \cap C\left(
\mathbb{R}
,X\right)  $ the Fr\'{e}chet space $L^{\infty}C$ endow with the $L_{loc}%
^{\infty}-$topology, wich is the topology of uniform convergence on bounded sets.

Let $E$ be a complete metric space with metric $d$. We write $B_{r}$ for the
open ball centre $0\in E$ and radius $r$. The following quantity is called the
Hausdorff (non-symmetric) semidistance from a set $X$ to a set $Y$ in a Banach
space $E$
\[
dist_{E}\left(  X,Y\right)  =\sup\limits_{x\in X}\inf\limits_{y\in
Y}\left\Vert .\right\Vert _{E}.
\]
Let $M$ be a subset of $E$ and let $%
\mathbb{R}
^{+}=\left[  0,\infty\right)  $. A mapping $\sigma=\sigma\left(  u,t\right)
$, where $\sigma:M\times\left[  0,\infty\right)  \rightarrow M$ is said to be
a semiflow on $M$ provided the following hold\newline1) $\sigma\left(
w,0\right)  =w$, for all $w\in M.$\newline2) The semigroup property holds, i.
e,%
\[
\sigma\left(  \left(  w,s\right)  ,t\right)  =\sigma\left(  w,s+t\right)
\text{ for all }w\in M\text{ and }s\text{, }t\in%
\mathbb{R}
^{+}.
\]
3) The mapping $\sigma:M\times\left(  0,\infty\right)  \rightarrow M$ is continuous.

If in addition the mapping $\sigma:M\times\left[  0,\infty\right)  \rightarrow
M$ is continuous we will say that the semiflow is continuous at $t=0$. Here we
use $t>0$ in order that the Robustness Theorem 23.14 in \cite{18} is valid,
see Sell \cite{18} and Hale \cite{8a}. For \ any $u\in M$ the positive
trajectory through $u$ is defined as the set $\gamma^{+}\left(  u\right)
=\left\{  \sigma\left(  t\right)  u,\text{ }t\geq0\right\}  $. For any set
$B\subset M$ we define the positive hull $\mathcal{H}^{+}\left(  B\right)  $
and the omega limit set $\omega\left(  B\right)  $ as follows%
\[
\mathcal{H}^{+}\left(  B\right)  =Cl_{M}\gamma^{+}\left(  B\right)  \text{ and
}\omega\left(  B\right)  =\cap_{\tau\geq0}\mathcal{H}^{+}\left(  \sigma\left(
\tau\right)  B\right)  .
\]
If $\mathcal{A}\subset E$ and $\varepsilon>0$ we write
\[
N_{\varepsilon}(\mathcal{A})=\{z\in E,\text{ }\inf_{a\in\mathcal{A}}d\left(
z,a\right)  <\varepsilon\}.
\]
for the open $\varepsilon-$neighbourhood of $\mathcal{A}$.

We denote by $A$ the Stokes operator $Au=-\triangle u$ for $u\in D\left(
A\right)  .$We recall that the operator $A$ is a closed positive self-adjoint
unbounded operator, with\newline$D\left(  A\right)  =\left\{  u\in
V_{0}\text{, }Au\in V_{0}\right\}  $. We have in fact, $D\left(  A\right)
=V_{2}$. The spectral theory of $A$ allows us to define the powers $A^{l}$ of
$A$ for $l\geq1$, $A^{l}$ is an unbounded self-adjoint operator in $V_{0}$
with a domain $D(A^{l})$ dense in $V_{2}\subset V_{0}$. We set here%
\[
A^{l}u=\left(  -\triangle\right)  ^{l}u\text{\ for }u\in D\left(
A^{l}\right)  =V_{2l}\text{.}%
\]
The space $D\left(  A^{l}\right)  $ is endowed with the scalar product and the
norm
\[
\left(  u,v\right)  _{D(A^{l})}=(A^{l}u,A^{l}v),\left\Vert u\right\Vert
_{D(A^{l})}=\{\left(  u,u\right)  _{D(A^{l})}\}^{\frac{1}{2}}.
\]
Now define the trilinear form $b(.,.,.)$ associated with the inertia terms
\[
b\left(  u,v,w\right)  =\sum_{i,j=1}^{3}%
{\displaystyle\int\limits_{\Omega}}
u_{i}\frac{\partial v_{j}}{\partial x_{_{i}}}w_{j}dx
\]
Recall that for $u$ satisfying $\nabla.u=0$ we have%
\begin{equation}
b\left(  u,u,u\right)  =0\text{ and }b\left(  u,v,w\right)  =-b\left(
u,w,v\right)  \text{.} \label{a1}%
\end{equation}
Hereafter, $c_{i}\in%
\mathbb{N}
$ ,will denote a dimensionless scale invariant positive constant which might
depend on the shape of the domain. The trilinear form $b\left(  .,.,.\right)
$ is continuous on $V_{m_{1}}\left(  \Omega\right)  \times V_{m_{2}+1}\left(
\Omega\right)  \times V_{m_{3}}\left(  \Omega\right)  $, $m_{i=1,2,3}\geq0$
\begin{equation}
\left\vert b\left(  u,v,w\right)  \right\vert \leq c_{0}\left\Vert
u\right\Vert _{m_{1}}\left\Vert v\right\Vert _{m_{2}+1}\left\Vert w\right\Vert
_{m_{3}}\text{ \ },\text{ }m_{3}+m_{2}+m_{1}\geq\frac{3}{2} \label{2c}%
\end{equation}
see \cite{7, 18}. The continuity property of the trilinear form enables us to
define (using Riesz representation theorem) a bilinear continuous operator
$B\left(  u,v\right)  $; $V_{2}\times V_{2}\rightarrow V_{2}^{\prime}$ will be
defined by
\[
\langle B\left(  u,v\right)  ,w\rangle=b\left(  u,v,w\right)  ,\text{ }\forall
w\in V_{2}\text{.}%
\]
We recall some inequalities that we will be using in what follows.

Young's inequality%
\begin{equation}
ab\leq\frac{\epsilon}{p}a^{p}+\frac{1}{q\epsilon^{\frac{q}{p}}}b^{q}%
,a,b,\epsilon>0,p>1,q=\frac{p}{p-1}. \label{a2}%
\end{equation}
Poincar\'{e}'s inequality%
\begin{equation}
\lambda_{1}\left\Vert u\right\Vert ^{2}\leq\Vert u\Vert_{1}^{2}\text{\ for all
}u\in V_{0}\text{,} \label{2}%
\end{equation}
where $\lambda_{1}$ is the smallest eigenvalue of the Stokes operator $A$.

\section{Navier-Stokes equations}

The conventional Navier-Stokes system can be written in the evolution form%
\begin{equation}%
\begin{array}
[c]{c}%
\dfrac{\partial u}{\partial t}+\nu Au+B\left(  u,u\right)  =f,\text{ }t>0,\\
\text{div }u=0,\text{ in }\Omega\times\left(  0,\infty\right)  \text{ and
}u\left(  x,0\right)  =u_{0},\text{ in }\Omega\text{,\ }%
\end{array}
\label{3}%
\end{equation}

let $f\in L^{\infty}\left(  0,\infty;V_{0}\right)  $ be given. We will say
that a function $u$ is a weak solution of the 3D Navier-Stokes of Class $LH$
(Leray--Hopf ) on $\left[  0,\infty\right)  $ provided that $u\left(
x,0\right)  =\nolinebreak u_{0}\left(  x\right)  \in V_{0}$, and the following
properties hold

1) $u\in L^{\infty}\left(  0,\infty;V_{0}\right)  \cap L_{loc}^{2}\left[
0,\infty;V_{1}\right)  $.

2) $\dfrac{du}{dt}\in\lbrack L_{loc}^{\frac{4}{3}}0,\infty;V_{1}^{^{\prime}})$.

Taking the inner product of (\text{%
\ref{3}
)} with $u$, and using (\text{%
\ref{a2}
)} we have%
\begin{equation}
\frac{d}{dt}\left\Vert u\left(  t\right)  \right\Vert ^{2}+2\nu\left\Vert
\nabla u\right\Vert ^{2}=2\left\langle f,u\right\rangle . \label{4}%
\end{equation}
by application of Young's inequality and the Poincar\'{e}'s Lemma, yields%
\begin{equation}
\frac{d}{dt}\left\Vert u\left(  t\right)  \right\Vert ^{2}+\nu\left\Vert
\nabla u\right\Vert ^{2}\leq\frac{\left\Vert f\right\Vert ^{2}}{\nu\lambda
_{1}}, \label{5}%
\end{equation}
using the Poincar\'{e} Lemma and Gronwall's inequality, to get%
\[
\left\Vert u\left(  t\right)  \right\Vert ^{2}\leq e^{-\nu\lambda_{1}\left(
t-t_{0}\right)  }\left\Vert u\left(  t_{0}\right)  \right\Vert ^{2}+\frac
{1}{\nu^{2}\lambda_{1}^{2}}\left\Vert f\right\Vert ^{2}\left(  1-e^{-\nu
\lambda_{1}\left(  t-t_{0}\right)  }\right)  \text{,with }0<t_{0}<t,
\]
3) which implies that%
\begin{equation}
\left\Vert u\left(  t\right)  \right\Vert ^{2}\leq e^{-\nu\lambda_{1}\left(
t-t_{0}\right)  }\left\Vert u\left(  t_{0}\right)  \right\Vert ^{2}+\frac
{1}{\nu^{2}\lambda_{1}^{2}}\left\Vert f\right\Vert ^{2}. \label{6}%
\end{equation}
Integrating (%
\ref{4}
) over $\left[  t_{0},t\right]  $ we find that%
\begin{equation}
\left\Vert u\left(  t\right)  \right\Vert ^{2}+2\nu\int_{t_{0}}^{t}\Vert
A^{\frac{1}{2}}u\left(  s\right)  \Vert^{2}ds\leq\left\Vert u\left(
t_{0}\right)  \right\Vert ^{2}+2\int_{t_{0}}^{t}\left\langle f\left(
s\right)  ,u\left(  s\right)  \right\rangle ds. \label{7}%
\end{equation}
4) The function $u$ satisfies the following equality%
\begin{equation}
\left\langle u\left(  t\right)  -u\left(  t_{0}\right)  ,v\right\rangle
+\nu\int_{t_{0}}^{t}\langle A^{\frac{1}{2}}u\left(  s\right)  ,A^{\frac{1}{2}%
}v\rangle ds+\int_{t_{0}}^{t}\langle B\left(  u\left(  s\right)  ,u\left(
s\right)  \right)  ,v\rangle ds=\int_{t_{0}}^{t}\left\langle f,v\right\rangle
ds, \label{8}%
\end{equation}
for all $v\in V_{1}$ and for all $t\geq t_{0}\geq0$.

The proof of the following theorem is given in \cite{13, 18, 21}.

\begin{theorem}
\label{Theorem01}Let $f\in V_{1}^{\prime}$ and $u_{0}\in V_{0}$ be given. Then
for every $T>0$, there exists a weak solution $u\left(  t\right)  $ of (%
\ref{3}
)\ from the space $L^{2}(0,T;V_{1})\cap L^{\infty}\left(  0,T;V_{0}\right)  $,
such that $u\left(  x,0\right)  =u_{0}$ and $u\left(  t\right)  $ satisfies
the energy equality (%
\ref{8}
).
\end{theorem}

Moreover (see \cite{21}), $u(.)$ is weakly continuous from $\left[
0,T\right]  $\ into $V_{0}$, the function $u\in C_{w}\left(  \left[
0,T\right]  ;\text{ }V_{0}\right)  $ and consequently $u\left(  x,0\right)
=u_{0}\left(  x\right)  \in V_{0}.$ Let $W$ is the set of all Leray--Hopf weak
solutions $u\left(  .\right)  $ of equation (%
\ref{3}
) in the space $L^{\infty}\left(  0,\infty;V_{0}\right)  \cap L_{loc}%
^{2}[0,\infty;V_{1})$ that satisfy the following properties

\begin{itemize}
\item $\dfrac{du}{dt}\in L_{loc}^{\frac{4}{3}}(0,\infty;V_{1}^{^{\prime}})$;

\item for almost all $t$ and $t_{0}$, with $t>t_{0}>0$, inequalities (%
\ref{7}
,%
\ref{8}
) are valid.
\end{itemize}

Let $X^{0}$ denote the Fr\'{e}chet space used to define the Leray-Hopf weak
solutions. Thus%
\[
\varphi\in X^{0}=L^{\infty}\left(  0,\infty;V_{0}\right)  \cap L_{loc}%
^{2}[0,\infty;V_{1}),
\]
where $\varphi\in C_{w}[0,\infty;V_{0})$ and we let $\mathfrak{F}^{0}$ denote
a compact, translation invariant set of forcing functions $f$ in%
\[
L^{\infty}C=L^{\infty}\left(
\mathbb{R}
,L^{2}\left(  \Omega\right)  \right)  \cap C\left(
\mathbb{R}
,L^{2}\left(  \Omega\right)  \right)
\]
where the topology on the Fr\'{e}chet space $L^{\infty}C$ is the topology of
uniform convergence on bounded sets in $%
\mathbb{R}
$.

Then, we use the Leray-Hopf solutions of the 3D Navier-Stokes equations with
$\varepsilon=0$ to generate a semiflow $\pi^{0}$ on $\mathfrak{F}^{0}\times
X^{0}$, where%
\[
\pi^{0}\left(  \tau\right)  \left(  f,\varphi\right)  =\left(  f_{\tau}%
,S^{0}\left(  f,\tau\right)  \varphi\right)  \text{ for }\tau\geq0,
\]
$f_{\tau}\left(  t\right)  =f\left(  \tau+t\right)  $ and $u\left(  t\right)
=S^{0}\left(  f,t\right)  \varphi$ is the Leray-Hopf solution of the 3D
Navier-Stokes equations that satisfies $u\left(  0\right)  =S^{0}\left(
f,0\right)  \varphi=\varphi\left(  0\right)  $. By using the theory of
generalized weak solutions, as in Sell \cite{17a} or \cite{18} , we note that
$\pi^{0}$ has a trajectory attractor $\mathfrak{A}_{0}\subset\mathfrak{F}%
^{0}\times X^{0}$ see Theorem 65.12 in \cite{18}, and Chepyzhov \cite{13a, 14}.

\section{The regularized Navier-Stokes system}

Using the operators defined in the previous section, we can write the modified
system (%
\ref{1}
) in the evolution form
\begin{equation}%
\begin{array}
[c]{c}%
\partial_{t}u^{\varepsilon}+\varepsilon A^{l}u^{\varepsilon}+\nu
Au^{\varepsilon}+B\left(  u^{\varepsilon},u^{\varepsilon}\right)  =f\left(
x\right)  ,\text{in }\Omega\times\left(  0,\infty\right) \\
\text{div }u^{\varepsilon}=0,\text{ in }\Omega\times\left(  0,\infty\right)
\text{, }u^{\varepsilon}\left(  x,0\right)  =u_{0}^{\varepsilon},\text{ in
}\Omega\text{.}%
\end{array}
\label{19}%
\end{equation}
The existence and uniqueness results for initial value problem $($\ref{1}$)$
can be found in J. L. Lions \cite[Remark 6.11]{13}.\newline In three
dimensions, the following theorem collects the main result in this work

\begin{theorem}
\label{Theorem03}For $l\geq\frac{5}{4}$, $\varepsilon>0$ fixed, $f\in
\nolinebreak L^{2}\left(  0,T;V_{0}^{\prime}\right)  $ and $u_{0}%
^{\varepsilon}\in V_{0}$ be given. There exists a unique weak solution of
$($\ref{19}$)$ which satisfies $u^{\varepsilon}\in L^{2}\left(  0,T;V_{l}%
\right)  \cap L^{\infty}\left(  0,T;V_{0}\right)  ,\forall T>0.$
\end{theorem}

We introduce the following result of the convergence of $u^{\varepsilon}$ as
the regularized parameter $\varepsilon\rightarrow$ $0$

\begin{theorem}
\label{Theorem03 copy(1)} For $l\geq\frac{3}{2}$, $\varepsilon>0$ fixed,
$f\in\nolinebreak L^{2}\left(  0,T;V_{0}^{\prime}\right)  $ and $u_{0}%
^{\varepsilon}\in V_{0}$ be given.\newline i) There exists a unique weak
solution of $($\ref{19}$)$ which satisfies%
\[
u^{\varepsilon}\in L^{2}(0,T;V_{l})\cap L^{\infty}\left(  0,T;V_{0}\right)
,\text{ }\forall T>0.
\]
\newline ii) This weak solution $u^{\varepsilon}$ converges strongly in
$L^{2}(0,T;V_{0})$ as $\varepsilon\rightarrow0$ to $u$ a weak solution of the
Navier-Stokes equations.
\end{theorem}

The above theorem is established directly by using of a general result
\cite[Theorem 3.9.]{23}. For $\varepsilon>0$, we let $\pi^{\varepsilon}$
denote the semiflow on $\mathfrak{F}^{0}\times X^{0}$ generated by the weak
solutions of regularized 3D Navier-Stokes equations of (%
\ref{19}
). Thus%
\begin{equation}
\pi^{\varepsilon}\left(  \tau\right)  =\left(  f_{\tau},S^{\varepsilon}\left(
f,\tau\right)  \varphi\right)  , \label{20}%
\end{equation}
where $u_{0}^{\varepsilon}=\varphi$ and%
\begin{equation}
u^{\varepsilon}\left(  t\right)  =S^{\varepsilon}\left(  f,t\right)
\varphi=S^{\varepsilon}\left(  f,t\right)  u_{0}^{\varepsilon} \label{21}%
\end{equation}
is the weak solution of (%
\ref{19}
) that satisfies $\varphi\left(  0\right)  =u_{0}^{\varepsilon}$.

Regarding the existence of the attractor $\mathfrak{A}_{\varepsilon}$ when
$\varepsilon>0$, we use especially the related papers of Chepyzhov and Vishik,
such as \cite{15,13a} to show that the system $\left(  \text{%
\ref{19}
}\right)  $ possesses a global attractor. For $\varepsilon>0$, we consider the
trajectory space $\mathcal{K}_{\varepsilon}$ of the modified Navier-Stokes
equations (%
\ref{19}
). $\mathcal{K}_{\varepsilon}$ is the union of all weak solutions
$u^{\varepsilon}\in X^{0}$ that satisfy (%
\ref{19}
), see (6.163) in \cite{13}.

Using the described scheme in \cite{15}, we construct the spaces
$\mathcal{S}_{b}$%
\[
\mathcal{S}_{b}=\{v\left(  .\right)  \in L^{\infty}\left(  0,T;V_{0}\right)
\cap L_{b}^{2}(0,T;V_{1}),\partial_{t}v\left(  .\right)  \in L_{b}%
^{2}(0,T;D(A^{\frac{l}{2}})^{\prime})\}
\]
with norm%
\[
\left\Vert v\right\Vert _{\mathcal{S}_{b}}=\left\Vert v\right\Vert _{L_{b}%
^{2}(0,T;V_{1})}+\left\Vert v\right\Vert _{L^{\infty}\left(  0,T;V_{0}\right)
}+\left\Vert \partial_{t}v\right\Vert _{L_{b}^{2}(0,T;D(A\frac{^{l}}%
{2})^{\prime})}%
\]
where
\[
\left\Vert v\right\Vert _{L_{b}^{2}(0,T;V_{1})}=\sup\limits_{t\geq0}(\int
_{t}^{t+1}\left\Vert v\left(  s\right)  \right\Vert _{1}^{2}ds)^{\frac{1}{2}%
},\left\Vert v\right\Vert _{L^{\infty}\left(  0,T;V_{0}\right)  }%
=\operatorname*{ess}\sup\limits_{t\geq0}\left\Vert v\right\Vert
\]
and
\[
\left\Vert \partial_{t}v\right\Vert _{L_{b}^{2}(0,T;V_{^{l}}^{\prime})}%
=\sup\limits_{t\geq0}(\int_{t}^{t+1}\left\Vert v\left(  s\right)  \right\Vert
_{V_{^{l}}^{\prime}}^{2}ds)^{\frac{1}{2}}.
\]
We need a topology in the space $\mathcal{K}_{\varepsilon}$. We define on
$X^{0}$ the following sequential topology which we denote $\Gamma$.

By definition, a sequence of functions $\{v_{n}\}\subseteq$ $X^{0}$ converges
to a function $v\in X^{0}$ in the topology $\Gamma$ as $n$ $\rightarrow\infty$
if, for any $T>0$, $v_{n}\rightarrow v$ weakly in $L^{2}(0,T;V_{1})$;
$v_{n}\rightarrow v$ weak-$\ast$\ in $L^{\infty}\left(  0,T;V_{0}\right)  $
and $v_{n}\rightarrow v$ strongly in $L^{2}(0,T;V_{0})$, as $n\rightarrow
\infty$.

We consider the topology $\Gamma$ on $\mathcal{K}_{\varepsilon}$. It is easy
to prove that the space $\mathcal{K}_{\varepsilon}$ is closed in $\Gamma$.
From the definition of $\mathcal{K}_{\varepsilon}$, it follows that\ $\pi
^{\varepsilon}\mathcal{K}_{\varepsilon}\subset\mathcal{K}_{\varepsilon}$ for
all $t\geq0$.

\begin{corollary}
\label{Proposition 3}If $u^{\varepsilon}\left(  t\right)  $ is a solution of (%
\ref{19}
), then the following inequalities hold for all $t>0$%
\begin{align}
\left\Vert u^{\varepsilon}\left(  t\right)  \right\Vert ^{2}  &  \leq
e^{-\nu\lambda_{1}t}\left\Vert u_{0}^{\varepsilon}\right\Vert ^{2}%
+\dfrac{\left\Vert f\right\Vert ^{2}}{\nu^{2}\lambda_{1}^{2}},\label{1y}\\
\int_{t}^{t+1}\left\Vert u^{\varepsilon}\left(  s\right)  \right\Vert ^{2}ds
&  \leq\dfrac{e^{-\nu\lambda_{1}t}}{\nu\lambda_{1}}\left\Vert u_{0}%
^{\varepsilon}\right\Vert ^{2}+\dfrac{\left\Vert f\right\Vert ^{2}}{\nu
^{2}\lambda_{1}^{2}},\label{2y}\\
\nu\int_{t}^{t+1}\left\Vert u^{\varepsilon}\left(  s\right)  \right\Vert
_{1}^{2}ds  &  \leq\frac{e^{-\nu\lambda_{1}t}}{\nu\lambda_{1}}\left\Vert
u_{0}^{\varepsilon}\right\Vert ^{2}+\frac{\left\Vert f\right\Vert ^{2}}%
{\nu^{2}\lambda_{1}^{2}}+\dfrac{\left\Vert f\right\Vert ^{2}}{\nu\lambda_{1}}.
\label{3y}%
\end{align}

\end{corollary}

\begin{proof}
Taking the inner product of (%
\ref{19}
) by $u^{\varepsilon}$, we obtain
\begin{equation}
\frac{d}{dt}\left\Vert u^{\varepsilon}\right\Vert ^{2}+2\varepsilon\Vert
A^{\frac{l}{2}}u^{\varepsilon}\Vert^{2}+2\nu\left\Vert \nabla u^{\varepsilon
}\right\Vert ^{2}=2\left(  f,u^{\varepsilon}\right)  \text{.} \label{8y}%
\end{equation}
Applying Young's inequality and using the Poincar\'{e} Lemma, we obtain%
\begin{equation}
\frac{d}{dt}\left\Vert u^{\varepsilon}\right\Vert ^{2}+\nu\left\Vert \nabla
u^{\varepsilon}\right\Vert ^{2}\leq\frac{\left\Vert f\right\Vert ^{2}}%
{\nu\lambda_{1}}. \label{4y}%
\end{equation}
Using the Gronwall's inequality over $\left[  0,\text{ }t\right]  $, we obtain
(%
\ref{1y}
). Integrating (%
\ref{1y}
) over $\left[  t,t+1\right]  $ we find (%
\ref{2y}
). Integrating (%
\ref{4y}
) over $\left[  t,t+1\right]  $ we find
\begin{equation}
\nu\int_{t}^{t+1}\left\Vert \nabla u^{\varepsilon}\left(  s\right)
\right\Vert ^{2}ds\leq\frac{\left\Vert f\right\Vert ^{2}}{\nu\lambda_{1}%
}+\left\Vert u^{\varepsilon}\left(  t\right)  \right\Vert ^{2}. \label{9y}%
\end{equation}
Applying inequality (%
\ref{1y}
), we get (%
\ref{3y}
).
\end{proof}

We recall the following result

\begin{lemma}
Let $f\in L^{2}\left(  0,T;V_{1}^{\prime}\right)  $, then, for any solution
$u^{\varepsilon}\left(  t\right)  $ of problem $($\ref{1}$)$ the time
derivative $\dfrac{du^{\varepsilon}}{dt}$ \textit{is uniformly bounded in}
$L^{2}\left(  0,T;V_{l}^{\prime}\right)  $.
\end{lemma}

A simple consequence of Lemma 3.6 \cite{23} is the following Corollary

\begin{corollary}
Let $f\in L^{2}\left(  0,T;V_{1}^{\prime}\right)  $, then any solution
$u^{\varepsilon}\left(  t\right)  $ of (%
\ref{19}
) satisfies
\begin{equation}
\int_{t}^{t+1}\left\Vert \partial_{t}u^{\varepsilon}\left(  s\right)
\right\Vert _{D(A^{\frac{l}{2}})^{\prime}}^{2}ds\leq C_{3}, \label{5y}%
\end{equation}
$C_{3}$ is a positive constant independent of $\varepsilon$ .
\end{corollary}

Moreover, due to estimates (%
\ref{1y}
) and (%
\ref{5y}
), we also have the uniform estimate.

\begin{proposition}
Let $f\in L^{2}\left(  0,T;V_{1}^{\prime}\right)  $, then any solution
$u^{\varepsilon}\left(  t\right)  $ of (%
\ref{19}
) satisfies the inequality%
\begin{equation}
\left\Vert \pi^{\varepsilon}\left(  u^{\varepsilon}\right)  \right\Vert
_{\mathcal{S}_{b}}^{2}\leq\dfrac{c_{7}e^{-\nu\lambda_{1}t}}{\nu\lambda_{1}%
}\left\Vert u^{\varepsilon}\left(  0\right)  \right\Vert ^{2}+\dfrac
{c_{7}\left\Vert f\right\Vert ^{2}}{\nu^{2}\lambda_{1}^{2}}+C_{4} \label{7y}%
\end{equation}
where the positive constant\ $C_{4}$ is independent of $\varepsilon$.
\end{proposition}

\begin{proposition}
For $l\geq\frac{3}{2}$ and $f\in L^{\infty}C$ a time independent functions,
$\pi^{\varepsilon}$ is a continuous family of semiflows on $\ X^{0}$.
\end{proposition}

\begin{proof}
Since $u^{\varepsilon}\in L^{2}\left(  0,T;V_{l}\right)  $ and $\dfrac
{du^{\varepsilon}}{dt}\in L^{2}\left(  0,T;V_{l}^{\prime}\right)  $,
$u^{\varepsilon}$ is almost everywhere equal to an uniform continuous function
from $[0,T]$ to the space\ $V_{0}$. The continuity of $u^{\varepsilon}$\ is a
direct consequence of \cite[Lemma 1.4. ChIII, Sec1]{21}.

From the result of the strong convergence, there exists $\varepsilon_{1}>0$,
such that
\begin{equation}
\left\Vert u^{\varepsilon}\left(  t\right)  -u^{\varepsilon_{0}}\left(
t\right)  \right\Vert \leq\epsilon,\forall\epsilon\geq0,\text{ for each
}\varepsilon\leq\varepsilon_{1},\label{06}%
\end{equation}
it follows from $\left(  \text{%
\ref{06}
}\right)  $ that $\lim\limits_{\varepsilon\rightarrow\varepsilon_{0}%
}\left\Vert \pi^{\varepsilon}\left(  t\right)  -\pi^{\varepsilon_{0}}\left(
t\right)  \right\Vert $ goes to $0$, for all $0\leq t\leq T$.

This shows that $\pi^{\varepsilon}$ is continuous semiflow on $X^{0}$ and
$\pi^{\varepsilon}$ approximates $\pi^{0}$ uniformly for $t$ in compact sets
in $\left[  0,\infty\right)  $.
\end{proof}

From Proposition 4.6 it follows that $\mathcal{K}_{\varepsilon}\subset
\mathcal{S}_{b}$ for all $\varepsilon>0$ and for all $\tau>0$. Also
Proposition 4.6 implies that the semigroup $\pi^{\varepsilon}$ has absorbing
set in $\mathcal{K}_{\varepsilon}$ for all $\varepsilon>0$ and for all
$\tau>0$ (We note, that this absorbing set does not depend on $\varepsilon$,
since the constant $C_{4}$ in (%
\ref{7y}
) is independent of $\varepsilon$), bounded in $\mathcal{S}_{b}$ and
inequality (%
\ref{7y}
) implies that absorbing set is compact in $\Gamma$. The continuity of
$\pi^{\varepsilon}$ is proved. These facts are sufficient to state that
$\pi^{\varepsilon}$ has a global attractor $\mathfrak{A}_{\varepsilon}$. Such
that $\mathfrak{A}_{\varepsilon}\subset\mathfrak{F}^{0}\times X^{0}$, bounded
in $\mathcal{S}_{b}$ and compact in $\Gamma$. For a more detailed, see
\cite{13a, 14, 15}.

\section{Upper semicontinuity of attractors}

We now prove the robustness property for the global attractor $\mathfrak{A}%
_{\varepsilon}$. We have shown in Proposition 4.7 the continuity of the family
of semiflows $\pi^{\varepsilon}$ on $X^{0}$. Having done this, We can simply
invoke Theorem 23.14 in \cite{18} to complete the proof of the robustness for
the family of attractors $\mathfrak{A}_{\varepsilon}$ at $\varepsilon=0$. We
denote by
\[
B_{R}=\{u^{\varepsilon}\left(  x,t\right)  ,0\leq t,0<\varepsilon<1;\left\Vert
\pi^{\varepsilon}\left(  u^{\varepsilon}\right)  \right\Vert _{\mathcal{S}%
_{b}}^{2}\leq R\}
\]
with $R=\dfrac{c_{7}e^{-\nu\lambda_{1}t}}{\nu\lambda_{1}}\left\Vert
u^{\varepsilon}\left(  0\right)  \right\Vert ^{2}+\dfrac{c_{7}\left\Vert
f\right\Vert ^{2}}{\nu^{2}\lambda_{1}^{2}}+C_{4}$, $u^{\varepsilon}\left(
x,t\right)  $ is a family of solutions of system (%
\ref{19}
), and the norms of $u^{\varepsilon}\left(  x,t\right)  $ in $\mathcal{S}_{b}$
are uniformly bounded, see Proposition 4.7. It is sufficient to show that a
small $\delta-$neighbourhood of attractor $\mathfrak{A}_{0}$ is an absorbing
set and $\pi^{\varepsilon}$ approximates $\pi^{0}$ on $B_{R}$ uniformly for
all $t$ in $\left[  0,\infty\right)  $, see \cite{8b, 18}.

\begin{theorem}
\label{main}For $l\geq\frac{3}{2}$, for $\varepsilon>0$ the family of
semiflows $\pi^{\varepsilon}$ generated by the weak solutions of the
regularized 3D Navier-Stokes equations (%
\ref{19}
) admits a compact attractor $\{\mathfrak{A}_{\varepsilon},$ $0<\varepsilon
\leq1\}$ which attracts bounded sets of $V_{0}$ and is contained in the
absorbing balls $B_{R}$ where $R$ is independent of $\varepsilon$.
Moreover,\ $d_{X^{0}}\left(  \mathfrak{A}_{\varepsilon},\mathfrak{A}%
_{0}\right)  \rightarrow0,$ as $\varepsilon\rightarrow0$.
\end{theorem}

\begin{proof}
Since $\mathfrak{A}_{0}$ is a global attractor, for any bounded set $B_{R_{0}%
}=\{u\left(  x,0\right)  \in V,\left\Vert u\left(  x,0\right)  \right\Vert
\leq R_{0}\}\subset V$, we have%
\begin{equation}
d_{X^{0}}\left(  \pi^{0}B_{R_{0}},\mathfrak{A}_{0}\right)  \rightarrow0\text{,
as }t\rightarrow\infty\text{.}\label{8h}%
\end{equation}
Thus, there exists $\delta>0$ such that%
\begin{equation}
d_{X^{0}}\left(  \pi^{0}B_{R_{0}},\mathfrak{A}_{0}\right)  \leq\frac{\delta
}{2},\text{ for }t\geq t_{\delta}\text{.}\label{7h}%
\end{equation}
Consequently%
\begin{equation}
\pi^{0}\left(  t\right)  B_{R_{0}}\subset N_{\delta}(\mathfrak{A}_{0})\text{,
for }t\geq t_{\delta}\text{,}\label{1h}%
\end{equation}
where $N_{\delta}(\mathfrak{A}_{0})$ be the $\delta$-neighborhood of
$\mathfrak{A}_{0}$. This shows that $N_{\delta}(\mathfrak{A}_{0})$ is an
absorbing set. Since $\pi^{\varepsilon}$ approximates $\pi^{0}$ uniformly for
all $t\geq0$, then for any $\delta>0$, there are $\varepsilon_{1}>0$ and
$t_{0}\geq0$ such that%
\begin{equation}
\pi^{\varepsilon}\left(  B_{R}\cap B_{R_{0}}\right)  \subset N_{\delta
}(\mathfrak{A}_{0})\text{, for }0<\varepsilon<\varepsilon_{1},\text{ }t\geq
t_{0}\text{.}\label{3h}%
\end{equation}
Since the attractor $\mathfrak{A}_{\varepsilon}\ $is contained in $B_{R}\cap
B_{R_{0}}$, we have%
\begin{equation}
\pi^{\varepsilon}(\mathfrak{A}_{\varepsilon})\subset N_{\delta}(\mathfrak{A}%
_{0})\text{, for }\varepsilon\leq\varepsilon_{1},\text{ }t\geq t_{0}%
\text{.}\label{4h}%
\end{equation}
Since $\mathfrak{A}_{\varepsilon}$ is an invariant set, we deduce that%
\begin{equation}
\mathfrak{A}_{\varepsilon}\subset N_{\delta}(\mathfrak{A}_{0})\text{, for
}0<\varepsilon<\varepsilon_{1},\text{ }t\geq t_{0}\text{.}\label{5h}%
\end{equation}
Moreover, since $\delta$ is arbitrary, we obtain the upper semicontinuity of
$\mathfrak{A}_{\varepsilon}$, at $\varepsilon_{0}=0$%
\begin{equation}
d_{X^{0}}\left(  \mathfrak{A}_{\varepsilon},\mathfrak{A}_{0}\right)
\rightarrow0,\text{ as }\varepsilon\in O\left(  \varepsilon_{0}\right)
.\label{6h}%
\end{equation}

\end{proof}

One can modify the argument described above so that the final result will have
broader applicability by allowing the family of forcing functions to vary with
$\varepsilon$, for $\varepsilon>0$. Thus, we consider the regularized
Navier-Stokes system (%
\ref{19}
) with a perturbed external force $f^{\varepsilon}$ in place of $f$, for
$\varepsilon>0$. Then (%
\ref{19}
) becomes%
\begin{equation}%
\begin{array}
[c]{r}%
\partial_{t}u^{\varepsilon}+\varepsilon A^{l}u^{\varepsilon}+\nu
Au^{\varepsilon}+B\left(  u^{\varepsilon},u^{\varepsilon}\right)
=f^{\varepsilon}\left(  x\right)  ,\text{in }\Omega\times\left(
0,\infty\right)  \\
\text{div }u^{\varepsilon}=0,\text{ in }\Omega\times\left(  0,\infty\right)
\text{, }u^{\varepsilon}\left(  x,0\right)  =u_{0}^{\varepsilon},\text{ in
}\Omega\text{.}%
\end{array}
\label{13y}%
\end{equation}
We show that the trajectory attractor of the perturbed system (%
\ref{13y}
) coincides with the trajectory attractor $\mathfrak{A}_{\varepsilon}$ of the
unperturbed system (%
\ref{3}
). Our results rely on the work of Hale (\cite{15}) who show that the limit
behaviour is valid even through $\mathfrak{F}^{\varepsilon}$, where
$\mathfrak{F}^{\varepsilon}$ denote a compact, translation invariant set of
perturbed forcing functions to vary with $\varepsilon$, for $\varepsilon>0$
and satisfy the condition%
\begin{equation}
\omega(\mathcal{H}^{+}\left(  f^{\varepsilon}\right)  )=\omega(\mathcal{H}%
^{+}\left(  f\right)  ).\label{14y}%
\end{equation}
Thus we would use $\mathfrak{F}^{\varepsilon}$ in place of $\mathfrak{F}^{0}$,
for $\varepsilon>0$. Moreover, by using a metric $d$ on the $L^{\infty}%
C$-toplogy, see \cite{18} for some samples, we can note that (%
\ref{14y}
) is equivalent to saying that for every $\delta>0$ there is an $\varepsilon
_{1}>0$ and $T_{\delta}=T\left(  \delta\right)  \geq0$ such that
\begin{equation}
d_{X^{0}}(f^{\varepsilon},\mathfrak{F}^{0})\leq\delta,\text{ for
}0<\varepsilon\leq\varepsilon_{2}\text{ and }f^{\varepsilon}\in\mathfrak{F}%
^{\varepsilon}\label{16y}%
\end{equation}
for any $t\geq T_{\delta}$, that is
\begin{equation}
\mathfrak{F}^{\varepsilon}\subset N_{\delta}\left(  \mathfrak{F}^{0}\right)
\text{, for }0<\varepsilon\leq\varepsilon_{2},\label{15y}%
\end{equation}
where $N_{\delta}$ denotes the $\delta$-neighborhood of $\mathfrak{F}^{0}$ in
$L^{\infty}C$. The resulting argument for robustness will then depend on two
parameters $\lambda=(\varepsilon,\delta)$, where $\lambda\rightarrow(0,0)$.

The following statement generalizes Theorem
\ref{main}

\begin{theorem}
Under the above conditions, the trajectory attractor of the perturbed 3D
Navier-Stokes system (%
\ref{13y}
) coincides with the trajectory attractor $\mathfrak{A}_{\varepsilon}$ of the
non-perturbed system (%
\ref{3}
). Moreover, the perturbed attractor of (%
\ref{13y}
) is upper semicontinuous with respect to $\varepsilon$ at $\varepsilon=0$.
\end{theorem}

\begin{proof}
The existence of trajectory attractor $\mathfrak{A}_{\varepsilon}$ is treated
above. The proof follows from formulas (%
\ref{14y}
), (%
\ref{15y}
) and Theorem
\ref{main}
.
\end{proof}

Proposition 4.7 can be used to extend the 2D result of Ou and
Sritharan\cite{16} to $l$-lplacian with $l>1$.

\end{document}